\documentclass[12pt,a4paper]{amsart}
\usepackage[top=3.5truecm, bottom=4.0truecm, left=3.5truecm, right=3.5truecm]{geometry}

\usepackage{mathtools}
\usepackage{amsmath , amssymb}
\usepackage{amsfonts}
\usepackage{amsthm}
\usepackage{graphicx}
\usepackage{subcaption}
\usepackage{tikz}
\usepackage{tikz-cd}
\usepackage{centernot}
\usepackage{here}

\DeclareMathOperator\im{Im}
\DeclareMathOperator\coker{coker}
\DeclareMathOperator\Int{Int}
\DeclareMathOperator\lk{lk}
\DeclareMathOperator\Hom{Hom}

\theoremstyle{definition}
\newtheorem{definition}{Definition}[section]
\newtheorem*{definition*}{Definition}
\newtheorem{theorem}{Theorem}
\newtheorem*{theorem*}{Theorem}
\newtheorem{proposition}[definition]{Proposition}
\newtheorem*{proposition*}{Proposition}
\newtheorem{corollary}[definition]{Corollary}
\newtheorem{lemma}[definition]{Lemma}
\newtheorem*{lemma*}{Lemma}

\newtheorem*{Note*}{Note}

\newtheorem*{example*}{Example}

\newcommand{\transpose}[1]{\,{\vphantom{#1}}^t\!{#1}}

\numberwithin{equation}{section}

\begin{document}

\title{HOMOLOGY OF RELATIVE TRISECTION AND ITS APPLICATION}
\author{HOKUTO TANIMOTO}
\date{}
\maketitle
\vspace{-2ex}
\begin{abstract}
Feller, Klug, Schirmer and Zemke showed the homology and the intersection form of a closed trisected 4-manifold are described in terms of trisection diagram. In this paper, it is confirmed that we are able to calculate those of a trisected 4-manifold with boundary in a similar way. Moreover, we describe a representative of the second Stiefel-Whitney class by the relative trisection diagram.
\end{abstract}
\vspace{1ex}

\section{Introduction}  %section 1
Gay and Kirby \cite{GK1} introduced a trisection as a decomposition of a 4-manifold into three 4-dimensional handlebodies. They mainly dealt with closed manifolds. Since Castro, Gay and Pinz\'on-Caicedo defined a relative trisection clearly in \cite{CGPC1} and \cite{CGPC2}, we are able to deal with the case of 4-manifolds with boundary in a similar way to the closed case such as a trisection diagram given by three families of curves on the central surface.

Feller, Klug, Schirmer and Zemke \cite{FKSZ1} expressd the homology and the intersection form of a closed 4-manifold in terms of trisection diagram in the way one of three families of curves plays a key role. On the other hand, Florens and Moussard \cite{FM1} described the homology such that roles of three families are symmetric.

 In this paper, we show that the homology and the intersection form of a trisected 4-manifold with boundary can be calculated in a similar way to \cite{FKSZ1} and \cite{FM1}. Our approaches derive from handle decompositions of the manifold associated with the relative trisection in both cases. Using a chain complex given by this approach, we express a representative of the second Stiefel-Whitney class in terms of relative trisection diagram. As a corollary, we obtain a necessary and sufficient condition for the existence of spin structures on a 4-manifold, which is described by curves on the central surface. 

\vspace{1ex}

This paper is organized as follows. In Section 2, we recall briefly definitions regarding relative trisections and their diagrams, and then state the main results of this paper. In Section 3, we consider two handle decompositions associated with the relative trisection:  one's union of a 0-handle and 1-handles is $X_1$ and the other's union is $D^2 \times \Sigma$. In Section 4, we review several facts of 3-dimensional topology. In Section 5, we compute the homology and the intersection form of $X$. As an application, we give a description of the second Stiefel-Whitney class in Section 6.

\vspace{1ex}

The author would like to thank Hisaaki Endo for many discussions and encouragement. 

\section{Main results}  %section 2
Let $X$ be a compact, connected, oriented, smooth 4-manifold with connected boundary. Relative trisections are more complicated than usual ones for closed 4-manifolds. We employ the same definition as \cite{CGPC1} and \cite{CGPC2}. Let $g, k= (k_1, k_2, k_3), p, b$ be integers with $g \ge p \ge 0$, $b \ge 1$ and $g+p+b-1 \ge k_i \ge 2p+b-1$, and put $l = 2p+b-1$. $D$, $\partial^0 D$ and $\partial^{\pm} D$ denote a third of a unit 2-dimensional disk, its arc and $\partial D \setminus \Int \partial^0 D$, respectively. Moreover, let $\partial^-D$ and $\partial^+D$ be each radius.
\begin{definition}
A $(g,k;p,b)$-{\it relative trisection} of $X$ is a decomposition $X = X_1 \cup X_2 \cup X_3$ such that:
\begin{enumerate}
\item[i)] $X_i \cap X_2 \cap X_3$ is diffeomorphic to $\Sigma$, a genus $g$ surface with $b$ boundary components;
   \vspace{0.5ex}
\item[ii)] $X_i \cap X_j = \partial X_i \cap \partial X_j$ for $i \neq j$, and each of them is diffeomorphic to a 3-dimensional compression body from $\Sigma$ to a genus $p$ surface $P$;
   \vspace{0.5ex}
\item[iii)] $X_i$ is diffeomorphic to $\natural^{k_i} S^1 \times D^3$;
   \vspace{0.5ex}
\item[iv)] $X_i \cap \partial X$ is diffeomorphic to $P \times \partial^0 D \cup  \partial P \times D$.
   \vspace{0.5ex}
\item[v)] $(X_{i-1} \cap X_i) \cup (X_i \cap X_{i+1})$ is a sutured Heegaard splitting of $P \times \partial^{\pm} D \, \sharp \, ( \sharp^{k_i-l} S^1 \times S^2)$.
\end{enumerate}
\end{definition}
We orient each compression body $X_i \cap X_{i-1}$ as a submanifold of $\partial X_i$,  surfaces $X_1 \cap X_2 \cap X_3$ and $(X_i \cap X_{i-1}) \cap \partial X$ as submanifolds of $\partial (X_i \cap X_{i-1})$. 

The above definition can be extended to a 4-manifold with more than one boundary components. In that case, our main results are also true. However, we assume that the  boundary is connected for simplicity. 

\vspace{1ex}

Same as Heegaard splittings and trisections, there is a diagram that determines the relative trisection. 
Let $\mu$, $\nu$ be a family of $g-p$ disjoint simple closed curves on $\Sigma$. For $(\Sigma; \mu, \nu)$ and $(\Sigma; \mu', \nu')$, we call that $(\Sigma; \mu, \nu)$ is {\it diffeomorphism and handle slide equivalent} to $(\Sigma; \mu', \nu')$ if they are related by a diffeomorphism between $\Sigma$ and a sequence of handle slides within each $\mu$ and $\nu$.
\begin{definition} \label{def:diagram}
A {\it $(g,k;p,b)$-relative trisection diagram} $(\Sigma; \alpha, \beta, \gamma)$ is a 4-tuple such that:
\begin{enumerate}
\item[i)] $\Sigma$ is a genus $g$ surface with $b$ boundary components;
   \vspace{0.5ex}
\item[ii)] each $\alpha , \beta$ and $\gamma$ is a family of $g-p$ disjoint simple closed curves on $\Sigma$;
   \vspace{0.5ex}
\item[iii)] each triple $(\Sigma; \alpha, \beta), \, (\Sigma; \beta, \gamma), \, (\Sigma; \gamma, \alpha)$ is diffeomorphism and handle slide equivalent to $(\Sigma; \delta^{k_i}, \epsilon^{k_i})$ shown in Figure 1, which is a standard sutured Heegaard diagram of $P \times \partial^{\pm} D \, \sharp \, ( \sharp^{k_i-l} S^1 \times S^2)$.
\end{enumerate}
\begin{figure}[H]
\centering
\includegraphics[keepaspectratio,scale=0.8,pagebox=cropbox]{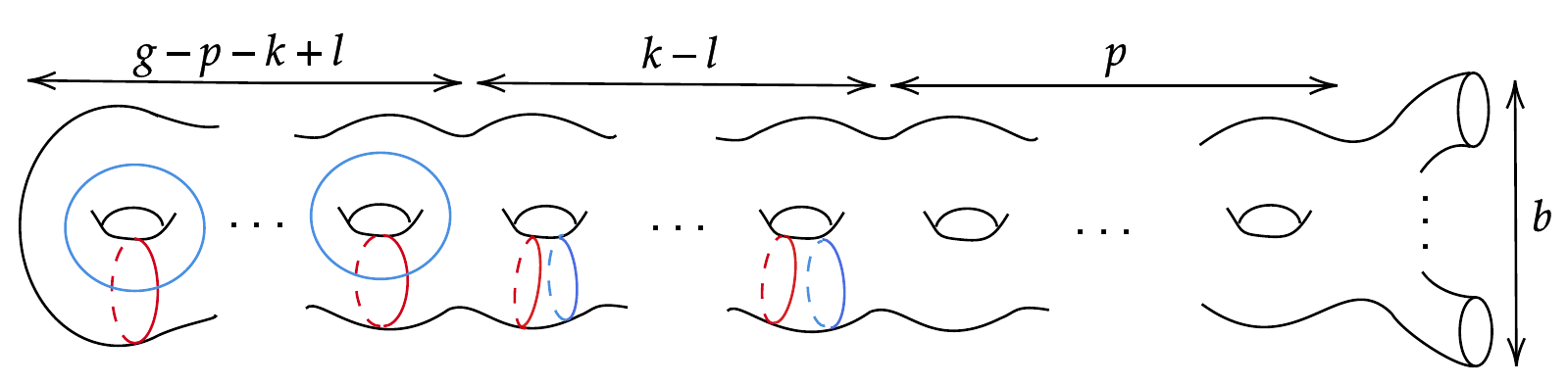}
\caption{A standard diagram $(\Sigma; \delta^{k}, \epsilon^{k})$, where the red curves are $\delta^k$, and blue curves are $\epsilon^k$.  }  %%%%Figure 1 of standard diagram
\end{figure}
\end{definition}
For $\nu \in \{ \alpha, \beta, \gamma\}$, let $C_{\nu}$ be a compression body given by attaching 3-dimensional 2-handles to $\Sigma \times I$ along a family of curves $\nu \times \{1\}$. Furthermore, $\Sigma_{\nu}$ denotes a surface given by performing surgeries to $\Sigma \times \{1\}$ along $\nu \times \{1\}$. A relative trisection is determined by the spine $C_{\alpha} \cup C_{\beta} \cup C_{\gamma}$, that is, a relative trisection diagram describes the relative trisection such that $X_1 \cap X_2 \cap X_3 = \Sigma$, $X_3 \cap X_1 = C_{\alpha}$, $X_1 \cap X_2 = C_{\beta}$ and  $X_2 \cap X_3 = C_{\gamma}$, respectively.

\vspace{1ex}

For each homomorphism induced by inclusions $\iota_{\nu} \colon H_1 (\Sigma) \to H_1(C_{\nu})$ and $\iota^{\partial}_{\nu} \colon \\ H_1(\Sigma , \partial \Sigma) \to H_1(C_{\nu} , \overline{\partial C_{\nu} \setminus \Sigma})$, let $L_{\nu}$ be $\ker \iota_{\nu}$, and $L^{\partial}_{\nu}$  $\ker \iota^{\partial}_{\nu}$. Using this subgroups, we calculate the homology in two ways.  

\begin{theorem} \label{theorem1}
The homology of $X$ can be obtained from the following chain complex $C^Y$: \vspace{-1ex}
{\small \begin{figure}[H]
\begin{tikzcd}
            0 \ar{r} &[-1em] (L_{\alpha} \cap L_{\gamma}) \oplus (L_{\beta} \cap L_{\gamma}) \ar{r}{\pi} &[-0.25em] L_{\gamma} \ar{r}{\rho} &[-0.25em] \Hom(L_{\alpha}^{\partial} \cap L_{\beta}^{\partial} , \mathbb{Z}) \ar{r}{0} &[-0.5em] \mathbb{Z} \ar{r} &[-1em] 0,
\end{tikzcd}
\end{figure}} \vspace{-2ex}
\noindent
where $\pi (x,y) = x+y$ and $\rho(x) = \langle - , x \rangle_{\Sigma}$.
\end{theorem}

\begin{theorem} \label{theorem2}
The homology of $X$ can also be obtained from the following chain complex $C^Z$: \vspace{-1ex}
{\small \begin{figure}[H]
\begin{tikzcd}
  0 \ar{r} &[-1.5em] (L_{\alpha} \cap L_{\beta}) \oplus (L_{\beta} \cap L_{\gamma}) \oplus (L_{\gamma} \cap L_{\alpha}) \ar{r}{\zeta} &[-1.2em] L_{\alpha} \oplus L_{\beta} \oplus L_{\gamma} \ar{r}{\iota} &[-1.2em] H_1(\Sigma) \ar{r}{0} &[-1.2em] \mathbb{Z} \ar{r} &[-1.5em] 0,
\end{tikzcd}
\end{figure}} \vspace{-2ex}
\noindent
where $\zeta (x, y, z) = (x-z, y-x, z-y)$ and $\iota$ is a homomorphsim induced by the inclusions $\iota_{\nu}$.
\end{theorem}

In both cases, $H_2(X)$ is isomorphic to the group described by $L_{\nu}$ in the same way as \cite{FKSZ1}. Moreover, the intersection form of $X$ is also described by this group and the intersection form of $\Sigma$. 
\vspace{1ex}

Let $\mu$, $\nu$ be families of oriented simple closed curves on $\Sigma$ or families of oriented arcs proper embedded in $\Sigma$. 
\begin{definition}
For $\mu$ and $\nu$, an intersection matrix $_{\mu}Q_{\nu}$ of $\mu$, $\nu$ is the matrix which consists of the intersection numbers $\langle [\mu_i] , [\nu_j] \rangle_{\Sigma}$. 
\end{definition}
For a relative trisection diagram $(\Sigma ; \alpha , \beta, \gamma)$, performing handleslides on $\alpha$ and  $\beta$, we can suppose by Definition \ref{def:diagram}.iii) that $(\alpha, \beta)$ is diffeomorphism equivalent to $(\delta^{k_1}, \epsilon^{k_1})$. Therefore, there exists a family of arcs $a$ such that $\{[\alpha] , [a]\}$ and $\{[\beta] , [a]\}$ are bases of $L_{\alpha}^{\partial}$ and $L_{\beta}^{\partial}$, respectively. We define several matrices to describe the second Stiefel-Whitney class as follows;
\begin{gather}
_{\gamma}Q_{\beta,\partial} = \left( \begin{array}{cc} _{\gamma}Q_{\beta} & _{\gamma}Q_a\\
 \end{array}\right) , \quad _{\partial}^{\alpha}Q_{\gamma} = \left( \begin{array}{c} _{\alpha}Q_{\gamma} \\ _a Q_{\gamma} \\ \end{array} \right), \notag \\[1ex]
R^g_{p,b} = I_{g-p} \oplus^p \begin{pmatrix} 0 & 0 \\ 1 & 0 \end{pmatrix} \oplus O_{b-1}. \notag
\end{gather}

\begin{theorem} \label{theorem3}
Suppose that $(\alpha,\beta)$ is diffeomorphism equivalent to $(\delta^{k_1} , \epsilon^{k_1})$. Then, as a representative of the second Stiefel-Whitney class $w_2$, we can obtain $c \colon L_{\gamma} \to \mathbb{Z}_2$ defined by the following formula with respect to the $\gamma$-basis
\[ c(x) = \sum_{i=1}^{g-p} \left( _{\gamma}Q_{\beta , \partial} \, \, R^g_{p,b} \hspace{1ex} \, ^{\alpha}_{\partial}Q_{\gamma} \right)_{i \, i} \,  x_i \pmod2 . \]
\end{theorem} 

Next, let $a$ be a family of arcs such that $\{[a]\}$ is a basis of $H_1(\Sigma , \partial \Sigma)$.  Regarding the complex $C^Z$, we can describe a representative by using the following matrix;
\begin{gather}
_{\nu}Q_a = \left( \begin{array}{cc} _{\alpha}Q_a \\ _{\beta}Q_a \\ _{\gamma}Q_a \end{array} \right), \quad _aQ_{\nu} = - ^t(_{\nu}Q_a), \notag \\[1ex]
S_{g,b} = \oplus^g \begin{pmatrix} 0 & 0 \\ 1 & 0 \end{pmatrix} \oplus O_{b-1}. \notag
\end{gather}

\begin{theorem} \label{theorem4}
As a representative of the second Stiefel-Whitney class, we can obtain $c' \colon L_{\alpha} \oplus L_{\beta} \oplus L_{\gamma} \to \mathbb{Z}$ defined by the following formula with respect to the $\alpha$-, $\beta$-, $\gamma$-basis
\[c'(x) = \sum_{i=1}^{3(g-p)} \left( _{\nu}Q_a \, \,  S_{g,b} \, \hspace{0.7ex} _aQ_{\nu} \right)_{i \, i} x_i \pmod{2}. \]
\end{theorem}
\vspace{1ex}

\section{Relative trisection and handle decomposition}   %section 3
According to \cite{GK1}, non-relative trisections induce handle decompositions of a 4-manifold. Therefore, trisection is considered as a variant of handle decomposition. We can calculate the homology and the intersection form of a 4-manifold by using a trisection. In this section, we show that relative trisection has similar properties. 
\vspace{1ex}

First, we introduce a handle decomposition constructed from $X_1$.
\begin{lemma} \label{lemma3.1}
For a trisected 4-manifold $X$ with connected boundary, there is a decomposition $X = Y \cup V_2 \cup V_3$ such that:
\begin{enumerate}
\item[(1)] each $V_2, V_3$ is diffeomorphic to $\natural^l S^1 \times D^3$, and those intersection $V_2 \cap V_3$ is an empty set;
    \vspace{0.5ex}
\item[(2)] $Y \cap V_i$ is included in $\partial V_i$, and is diffeomorphic to $P \times I$;
    \vspace{0.5ex}
\item[(3)] $Y$ is a handlebody that consists of one 0-handle, $k_1$ 1-handles, $g-p$ 2-handles, $k_2 + k_3 - 2l$ 3-handles;
    \vspace{0.5ex}
\item[(4)] the curves $\gamma$ are the attaching circles of 2-handles for $Y$, and their framings are induced by $\Sigma$;
    \vspace{0.5ex}
\item[(5)] $Y^1 = X_1$.
\end{enumerate}
\end{lemma}
\begin{proof}
Since $X_1$ is diffeomorphic to $\natural^{k_1} S^1 \times D^3$, we regard it as $Y^1$. 

$X_2 \cap X_3$ is obtained by attaching 3-dimensional 2-handles to $X_1 \cap X_2 \cap X_3 \cong \Sigma$ along  $\gamma$ with framing induced by $\Sigma$. We get 4-dimensional 2-handles by thickening  them. In other words, for the 3-dimensional attaching region $S^1 \times D^1 \subset \Sigma$ and an embedded $\Sigma \times D^1$ in $(X_3 \cap X_1) \cup (X_1 \cap X_2)$ such that $\Sigma \times D^1 \cap (X_3 \cap X_1) = \Sigma \times [-1,0]$ and $\Sigma \times D^1 \cap (X_1 \cap X_2) = \Sigma \times [0,1]$, the 4-dimensional attaching region is $S^1 \times D^1 \times D^1$. A surgery derived from this attaching 2-handles changes $C_{\alpha} \cup C_{\beta}$ into $(C_{\alpha} \cup C_{\gamma}) \cup P \times D^1 \cup (C_{\gamma} \cup C_{\beta})$.
$X$ is obtained by attaching $\natural^{k_3} S^1 \times D^3$, $\natural^{k_2} S^1 \times D^3$ on it along $C_{\alpha} \cup C_{\gamma}$, $C_{\gamma} \cup C_{\beta}$ which are diffeomorphic to $P \times \partial^{\pm} D \, \sharp \, ( \sharp^{k_i-l} S^1 \times S^2)$.

We separate $\natural^{k_i} S^1 \times D^3$ into $\natural^{k_i-l} S^1 \times D^3$ included in $\Int X$ and $\natural^l S^1 \times D^3$. We regard the former as a 4-dimensional handlebody which consists of one 0-handle and $k_i -l$ 1-handles, and the boundary sum between the former and the latter as connecting them with a 1-handle. Considering the dual decomposition, attaching the former is the same as attaching $k_i-l+1$ 3-handles and one 4-handle. It is clear that the pair of a 3-handle of boundary sum and the 4-handle is a canceling pair. In this way, we obtain the handlebosy $Y$ satisfying (3), (4) and (5). Each $V_2$, $V_3$ should be $\natural^l S^1 \times D^3$ attached along $C_{\gamma} \cup C_{\beta}$, $C_{\alpha} \cup C_{\gamma}$.
\end{proof}

On the other hand, there is another handle decomposition constructed from $\Sigma \times D^2$.
\begin{lemma}
For a trisected 4-manifold $X$ with connected boundary, there is a decomposition $X = Z \cup W_1 \cup W_2 \cup W_3$ such that:
\begin{enumerate}
\item[(1)] $W_i$ is diffeomorphic to $\natural^l S^1 \times D^3$, and those intersection $W_i \cap W_j \,\, (i \neq j)$ is an empty set;
   \vspace{0.5ex}
\item[(2)] $Z \cap W_i$ is included in $\partial W_i$, and is diffeomorphic to $P \times I$;
   \vspace{0.5ex}
\item[(3)] $Z$ is a handlebody that consists of one 0-handle, $2g+b-1$ 1-handles, $3(g-p)$ 2-handles, $\sum_i k_i - 3l$ 3-handles;
   \vspace{0.5ex}
\item[(4)] the curves $\alpha$, $\beta$, $\gamma$ are the attaching circles of 2-handles for $Z$. Furthermore, their framings are induced by $\Sigma$;
   \vspace{0.5ex}
\item[(5)] $Z^1 \cong \Sigma \times D^2$.
\end{enumerate}
\end{lemma}
\begin{proof}
Since $\Sigma \times D^2$ is diffeomorphic to $\natural^{2g+b-1} S^1 \times D^3$, we regard it as $Z^1$. Considering how to reconstruct a trisection from a diagram, we can obtain the decomposition in a similar way to the proof of Lemma \ref{lemma3.1}.
\end{proof}
\begin{figure}[H]
\centering
\includegraphics[keepaspectratio,scale=0.7,pagebox=cropbox]{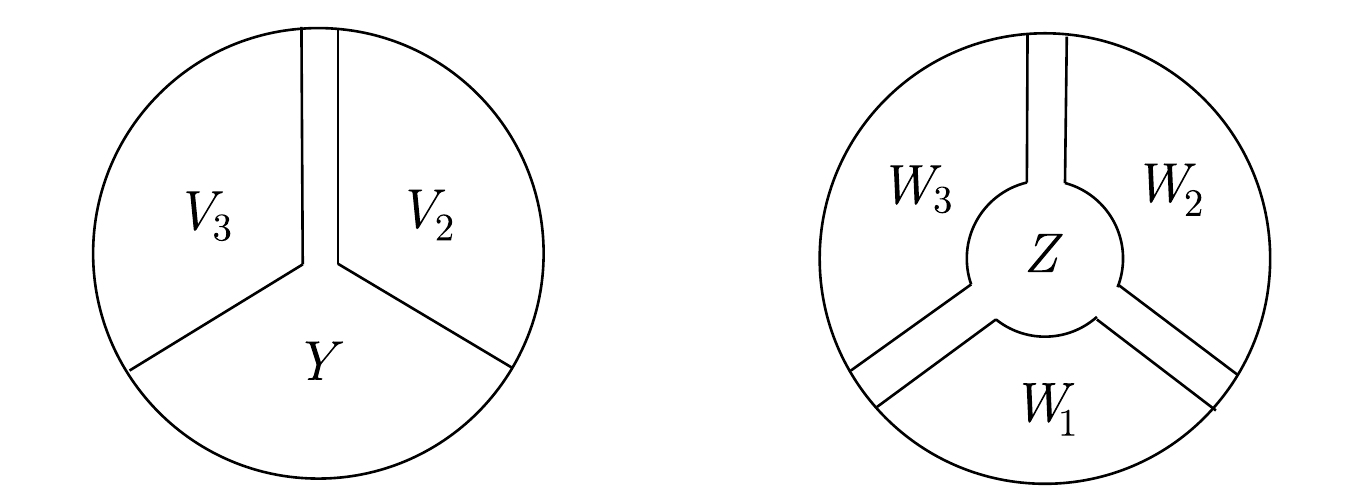}
\caption{Two decompositions of $X$: $Y \cup V_2 \cup V_3$ and $Z \cup W_1 \cup W_2 \cup W_3$. }
\end{figure}

According to the following lemma, $X$ is homotopy equivalent to each of $Y$ and $Z$. Therefore, the homology and the intersection form of $X$ can be calculated from the handlebodies $Y$ and $Z$.
\begin{lemma}
There are deformation retracts $r \colon (X, \partial X) \to (Y, \partial Y)$ and $r' \colon \\
(X, \partial X) \to (Z, \partial Z)$. Consequently, $H_i(X)$ is isomorphic to $H_i(Y)$ and $H_i(Z)$, and $H^*(X, \partial X)$ is isomorphic to each of $H^*(Y, \partial Y)$ and $H^*(Z, \partial Z)$ as rings.
\end{lemma}
\begin{proof}
Since $V_i \cong \natural^l S^1 \times D^3$ and $(Y \cap V_i ) \, \cup \, \overline{\partial V_i \setminus (Y \cap V_i)}$ is a standard Heegaard splitting of $\partial V_i \cong \sharp^l S^1 \times S^2$, we can construct deformation retracts $r_i \colon V_i \to Y \cap V_i$. Using this $r_i$ and $id_Y$, we obtain the desired map $r$. The construction of $r'$ is the same.
\end{proof}
\begin{figure}[H]
\centering
\includegraphics[scale=0.8,pagebox=cropbox]{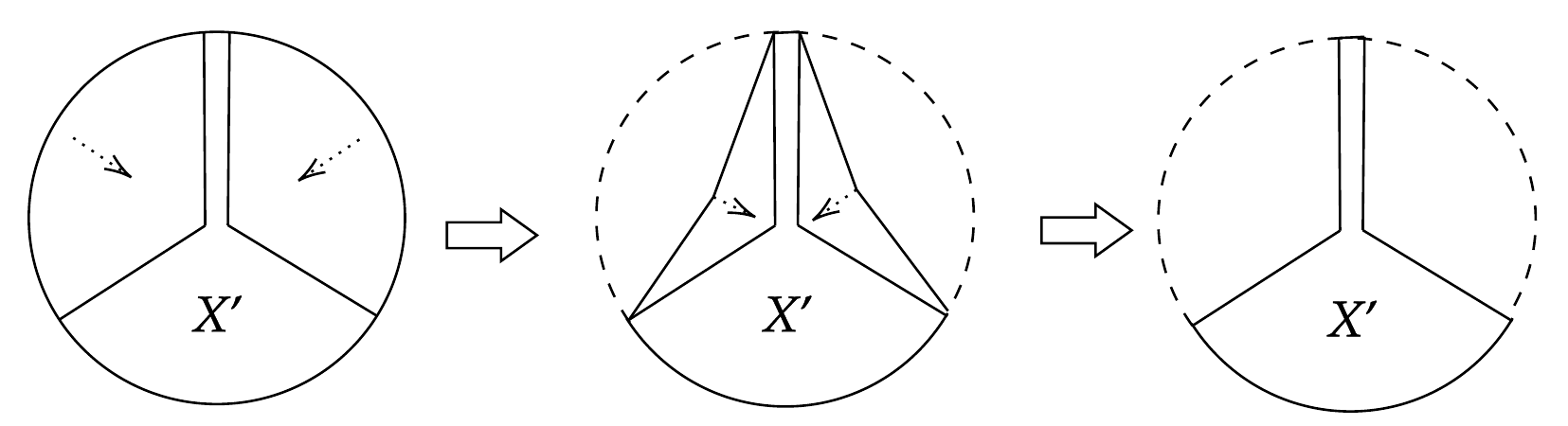}
\caption{$r \colon (X , \partial X) \to (Y , \partial Y)$.}
\end{figure}

\section{Facts of 3-dimensional topology} %section 4
To consider the homology and the intersection of $Y$ and $Z$, we review topology of 3-manifolds. We modify some lemmas in \cite{FKSZ1} to obtain the lemmas in this section, which can be applied to sutured Heegaard splitting. We omit the proofs of several lemmas in this paper since they are similar to those of \cite{FKSZ1}.
\begin{lemma} \label{lemma4.1}
Let $C_{\nu}$ be a relative compression body with surface $\Sigma$. Then:
\begin{enumerate}
\item[(1)] $\{ [\nu_1], \dots , [\nu_{g-p}] \}$ forms a basis of $L_{\nu}$;
   \vspace{0.5ex}
\item[(2)] there is a family of arcs $a$ such that $\{ [\nu_1], \dots , [\nu_{g-p}] , [a_1] ; \dots , [a_l] \}$ forms a basis of $L_{\nu}^{\partial}$,
   \vspace{0.5ex}
\item[(3)] $\{\, x \in H_1(\Sigma) \mid \forall y \in L_{\nu}^{\partial} \, , \, \langle y , x \rangle_{\Sigma} = 0 \,\} = L_{\nu}$.
\end{enumerate}
\end{lemma}
\begin{proof}
We consider the exact sequence
{\small \begin{figure}[H]
\vspace{-1ex} \begin{tikzcd}
 H_2(C_{\nu})  \ar{r} &  H_2(C_{\nu} , \Sigma)  \ar{r}{\partial} & H_1(\Sigma) \ar{r}{\iota_{\nu}} & H_1(C_{\nu})
\end{tikzcd}
\end{figure}} \vspace{-1.5ex}
\noindent
of $(C_{\nu}, \Sigma)$ to show Claim (1).
$H_2(C_{\nu}) = 0$ since $C_{\nu}$ is diffeomorphic to $\natural^{g+p+b-1} S^1 \times D^2$. The core disks of 2-handles form a basis of $H_2(C_{\nu}, \Sigma)$, and the image of each core disk is its attaching circle. These prove Claim (1) because of the exactness.
\vspace{1ex}

To consider exact sequences of triples $(C_{\nu}, \Sigma, \partial \Sigma)$ and $(C_{\nu}, \overline{\partial C_{\nu} \setminus \Sigma}, \partial \Sigma)$,  we deal with the sequence of $(C_{\nu}, \partial \Sigma)$:
\vspace{-1ex} {\small \begin{figure}[H]
\begin{tikzcd}
0 \ar{r} &[-1.5em] H_2(C_{\nu} , \partial \Sigma) \ar{r} &[-1.5em] H_1(\partial \Sigma) \ar{r}{\sigma_1} &[-1.2em] H_1(C_{\nu})  \ar{r} &[-1.5em] H_1(C_{\nu} , \partial \Sigma) \ar{r} &[-1.5em] H_0(\partial \Sigma) \ar{r}{\sigma_0} &[-1.2em] H_0(C_{\nu}).
\end{tikzcd}
\end{figure}} \vspace{-2ex} \noindent
Let $c_0, \dots , c_{b-1}$ be boundary components of $\partial \Sigma$, and then the homomorphism $\sigma_1$ satisfies $\sigma_1([c_0]) = -[c_1]-\cdots-[c_{b-1}]$, $\sigma_1([c_i]) = [c_i] \, \, (i = 1, \dots , b-1)$. Since $\ker \sigma_1$ is the free submodule which has one generator $[c_0]+\cdots+[c_{b-1}]$, $[\Sigma]$ forms a basis of $H_2(C_{\nu}, \partial \Sigma)$. Each $c_i \subset \natural^{g-p+b-1} S^1 \times D^2$ can be regarded as $S^1 \times \{*_i\}$, and those homology classes form a basis of $\im \sigma_1$. Therefore, $\coker \sigma_1$ is a free module. It is easy to see that $\ker \sigma_0$ is free. These prove that $H_1(C_{\nu}, \partial \Sigma)$ is also free.

Since the image of $[\Sigma]$ by the homomorphism $H_2(C_{\nu}, \partial \Sigma) \to H_2(C_{\nu}, \Sigma)$ is zero, we obtain the following diagram:
{\small \begin{figure}[H]
\centering
\begin{tikzcd}
 0 \ar{r} &[-2ex] H_2(C_{\nu} , \Sigma) \ar{r}{\partial} &[-1ex] H_1(\Sigma , \partial   \Sigma) \ar{d}[swap]{\eta} \ar{rd}{\iota^{\partial}_{\nu}} \ar{r}{\eta} &[-1ex] H_1(C_{\nu} , \partial \Sigma) \ar{d}{\theta} \ar{r} &[-2ex] 0 \\
 0 \ar{r} & H_1(\overline{\partial C_{\nu} \setminus \Sigma} , \partial \Sigma) \ar{r} & H_1(C_{\nu} , \partial \Sigma) \ar{r}[swap]{\theta} & H_1(C_{\nu} , \overline{\partial C_{\nu} \setminus \Sigma}) \ar{r} & 0.
\end{tikzcd}
\end{figure}} \vspace{-2ex}
\noindent
The upper row is the exact sequence of $(C_{\nu}, \Sigma, \partial \Sigma)$, and the lower row is the exact sequence of $(C_{\nu}, \overline{\partial C_{\nu} \setminus \Sigma}, \partial \Sigma)$. The upper sequence splits such that $H_1(\Sigma, \partial \Sigma) \simeq \im \partial \oplus H_1(C_{\nu}, \partial \Sigma)$, and $\iota^{\partial}_{\nu}(x,y) = \theta(y)$. Let $a$ be a family of $l$ proper arcs in $\Sigma \setminus \nu$ such that the surgery along $a$ changes $\Sigma_{\nu}$ into a disk. $\{ [a_1], \dots , [a_l] \}$ forms a basis of $H_1(\overline{\partial C_{\nu} \setminus \Sigma}, \partial \Sigma)$. Since we can regard $a \subset \Sigma$, so $\ker \iota^{\partial}_{\nu} = \im \partial \oplus \ker \theta$. This proves Claim (2).
\vspace{1ex}

We separate $\Sigma$ into $\Sigma_{g-p} \sharp \Sigma_{\nu}$ such that $\nu \subset \Sigma_{g-p}$ and $a \subset \Sigma_{\nu}$. Let $\nu^*$ be a family of $g-p$ curves in $\Sigma_{g-p}$ such that $\{ [\nu], [\nu^*] \}$ is a symplectic basis of $H_1(\Sigma_{g-p})$ regarding the intersection form of $\Sigma_{g-p}$. Let $a^*$ be a family of $l$ curves in $\Sigma_{\nu}$ as the following diagram. 
\begin{figure}[H]
\centering
\includegraphics[keepaspectratio,scale=0.9,pagebox=cropbox]{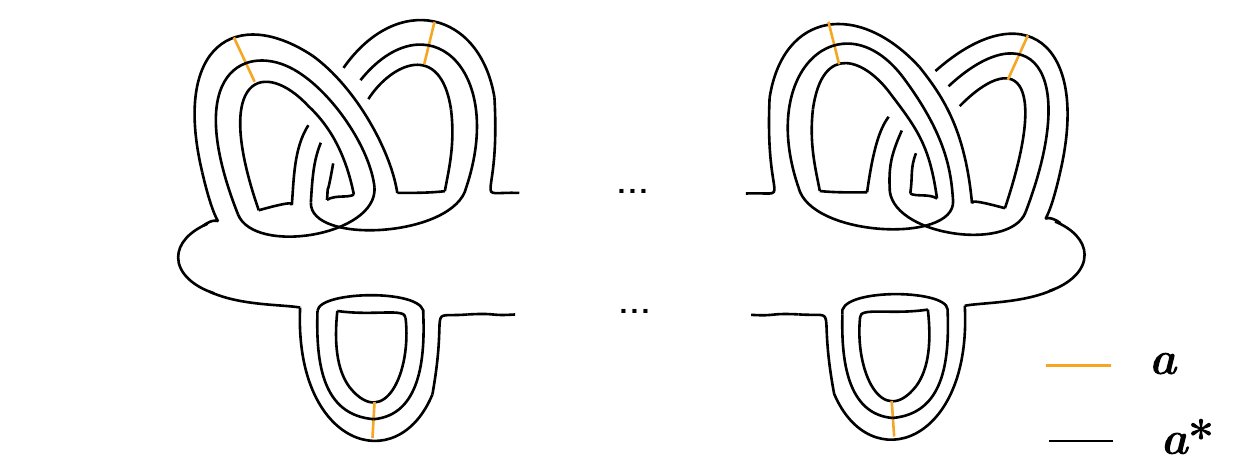}
\caption{The curves $a$ and $a^*$.}
\end{figure} \noindent
$\{[a^*]\}$ is a basis of $H_1(\Sigma_{\nu})$, and then $\{[\nu], [\nu^*], [a^*]\}$ is a basis of $H_1(\Sigma)$. For $x = (x_1, \dots ,x_{2g+b-1}) \in H_1(\Sigma)$, we have
\[ \langle \nu_i, x \rangle_{\Sigma} = x_{i+g-p}, \, \, \langle a_i, x \rangle_{\Sigma} = x_{i+2(g-p)}. \]
(1), (2) and this show that the left-hand side of (3) is $L_{\nu}$.
\end{proof}
Let $(L_{\nu}^{\partial})^{\bot}$ be the left-hand side of Lemma \ref{lemma4.1} (3), then $(L_{\nu}^{\partial})^{\bot} = L_{\nu}$ according to (3). We can regard $L_{\nu}$ as a submodule in $L_{\nu}^{\partial}$. 
\vspace{1ex}

Let $(\Sigma; C_{\mu}, -C_{\nu})$ be a sutured Heegaard splitting of a 3-manifold $M$. We define $\iota_M$ and $\iota_M^{\partial}$ in the same way as $\iota_{\nu}$ and $\iota_{\nu}^{\partial}$. We consider the Mayer-Vietoris sequences of $(\Sigma ; C_{\mu} , -C_{\nu})$ and $\left( (\Sigma , \partial \Sigma)\, ; (C_{\mu} , \overline{\partial C_{\mu} \setminus \Sigma}) \, , -(C_{\nu} , \overline{\partial C_{\nu} \setminus \Sigma}) \right)$. Suppose $Q \subset \Int M$ is an oriented embedded surface which is transverse to $\Sigma$ and represents  $x \in H_2(M)$. The image of the boundary homomorphism $\partial x$ is represented by the family of curves $Q \cap \Sigma$. We orient them as the boundary of the oriented surface $Q \cap C_{\mu}$. Note that the orientation induced from $Q \cap C_{\nu}$ is the opposite. The boundary homomorphism of $H_2(M, \partial M)$ is also similar.
\begin{lemma} \label{lemma4.2}
Let $\partial \colon H_2(M) \to H_1(\Sigma)$ and $\partial' \colon H_2(M, \partial M) \to H_1(\Sigma, \partial \Sigma)$ be the boundary homomorphisms coming from the Mayer-Vietoris sequences for the triple $(\Sigma ; C_{\mu} , -C_{\nu})$ and $\left( (\Sigma , \partial \Sigma)\, ; (C_{\mu} , \overline{\partial C_{\mu} \setminus \Sigma}) \, , -(C_{\nu} , \overline{\partial C_{\nu} \setminus \Sigma}) \right)$, respectively. Then:
\begin{enumerate}
\item[(1)] the maps $\partial$ and $\partial'$ are isomorphisms between $H_2(M)$ and $L_{\mu} \cap L_{\nu}$, and $H_2(M, \partial M)$ and $L_{\mu}^{\partial} \cap L_{\nu}^{\partial}$, respectively;
   \vspace{0.5ex}
\item[(2)] for any $z \in H_2(M, \partial M)$ and $x \in H_1(\Sigma)$, $\langle z , \iota_M x \rangle_M = \langle \partial' z , x \rangle_{\Sigma}$;
   \vspace{0.5ex}
\item[(3)] if $H_1(M, \partial M)$ is torsion free and $x \in H_1(\Sigma)$, then $x \in \ker \iota_M$ if and only if $x \in (L_{\mu}^{\partial} \cap L_{\nu}^{\partial})^{\bot}$.
\end{enumerate}
\end{lemma}
\begin{proof}
The proof of this lemma is similar to the proof of Lemma 2.2 \cite{FKSZ1}.
\end{proof}
\vspace{1ex}

The intersection form of a 4-manifold with handle decomposition can be calculated by using a linking matrix of the framed link obtained from attaching circles of the 2-handles. We can also define the linking number in the case of 3-manifolds with boundaries by using Seifert surfaces. In regard to a relationship between the linking number and the intersection on $\Sigma$, we obtain the following lemma.
\begin{lemma} \label{lemma4.3}
Suppose $J$ and $K$ are families of embedded oriented curves in $\Int \Sigma$ such that $\iota_M ([J]) =\iota_M ([K]) = 0$,$J \cap K = \emptyset$. Then we have:
\begin{enumerate}
\item[(1)] there is $j \in L_{\mu}$ such that $\langle y , j \rangle_{\Sigma} = \langle y , [J] \rangle_{\Sigma}$ for all $y \in L_{\nu}^{\partial}$;
   \vspace{0.5ex}
\item[(2)] for $j \in L_{\mu}$ satisfying Claim (1), $\lk(J,K)_M = \langle j, [K] \rangle_{\Sigma}$.
\end{enumerate}
\end{lemma}
\begin{proof}
The proof of this lemma is the same as the proof of Lemma 2.4 \cite{FKSZ1}.
\end{proof}
We call the closed version of this lemma also Lemma \ref{lemma4.3}.
\vspace{1ex}

Let $N$ and $M$ be $\sharp^k S^1 \times S^2$ and $P \times \partial^{\pm} D \, \sharp \, (\sharp^{k-l} S^1 \times S^2)$, respectively. Considering a genus $l$ Heegaard splitting of $\sharp^l S^1 \times S^2$ constructed from $P \times \partial^{\pm} D$ and $P \times \partial^0 D \cup  \partial P \times D$, we can regard $M$ as a submanifold of $N$. $H_2(N)$ is a free module which has a basis $\bigl\{ [\{ *_i \} \times S^2] \bigr\}$. Suppose that 
$a$ is the family of arcs on $P$ as in Figure 4, then $\bigl\{[a_1 \times \partial^{\pm} D] , \dots , [a_l \times \partial^{\pm} D] , \bigl[ \{ *_{l + 1} \} \times S^2 \bigr] , \dots , \bigl[ \{ *_k \} \times S^2 \bigr] \bigr\}$ is a basis of $H_2(M, \partial M)$. We define an isomorphism $b \colon H_2(N) \to H_2(M, \partial M)$ as follows:
\[ b ([\{ *_i \} \times S^2 ]) = \begin{cases}
                                             [a_i \times \partial^{\pm} D]  \, & (i \le l) \\[1ex]
                                             \bigl[ \{ *_i \} \times S^2 \bigr] \, & (l+1               \le i) .
                                         \end{cases} \]
In regard to the intersections in $N$ and $M$, we have:
\begin{lemma} \label{lemma4.4}
For all $z \in H_2(N)$ and $x \in H_1(\Sigma)$, $\langle z, \iota_N x \rangle_N = \langle b(z), \iota_M x \rangle_M$.
\end{lemma}
\begin{proof}
It is sufficient to show this regarding the basis $\{[\{*_i\} \times S^2]\}$. The case $l+1 \le i$ is clear. Considering the above canonical Heegaard splitting of $\sharp^l S^1 \times S^2$, we obtain $N$ by attaching a handlebody $H_l$ to $M$ along this splitting. Therefore, $\{*_i\} \times S^2$ is the sphere which is made of $a_i \times \partial^{\pm} D$ and a core disk of $H_l$, and then $a_i \times \partial^{\pm} D = (\{*_i\} \times S^2) \cap M$. Let $\gamma$ be a representative of $x$ such that it is transverse to $\{*_i\} \times S^2$. Since $\gamma$ is also in $M$, $(\{*_i\} \times S^2) \cap \gamma = (a_i \times \partial^{\pm} D) \cap \gamma$.
\end{proof}
\vspace{1ex}

\section{Calculating the homology and the intersection form}
Let $(\Sigma; X_1, X_2, X_3)$ be a $(g,k;p,b)$-relative trisection of a 4-manifold $X$. $M_{\mu\nu}$ denotes the manifold which admits a sutured Heegaard splitting $(\Sigma; C_{\mu}, -C_{\nu})$, and let $\partial_{\mu\nu}, \, \partial'_{\mu\nu}$ be the maps in Lemma \ref{lemma4.2} for this $M_{\mu\nu}$. In this section, we show how the homology and the intersection form of $X$ are calculated in terms of the intersection form of $\Sigma$ and the subgroups $L_{\nu}$. 

\setcounter{theorem}{0}
\begin{theorem}
The homology of $X$ can be obtained from the following chain complex $C^Y$: \vspace{-1ex}
{\small \begin{figure}[H]
\begin{tikzcd}
            0 \ar{r} &[-1em] (L_{\alpha} \cap L_{\gamma}) \oplus (L_{\beta} \cap L_{\gamma}) \ar{r}{\pi} &[-0.25em] L_{\gamma} \ar{r}{\rho} &[-0.25em] \Hom(L_{\alpha}^{\partial} \cap L_{\beta}^{\partial} , \mathbb{Z}) \ar{r}{0} &[-0.5em] \mathbb{Z} \ar{r} &[-1em] 0,
\end{tikzcd}
\end{figure}}\vspace{-2ex} \noindent
where $\pi (x,y) = x+y$ and $\rho(x) = \langle - , x \rangle_{\Sigma}$.
\end{theorem}
\begin{proof}
We calculate the homology of the handlebody $Y$. Let $(C_*, \partial_*)$ be the chain complex of $Y$. Since our manifold is oriented and connected, we will focus on the following nontrivial part of $(C_*, \partial_*)$: \vspace{-1ex}
\begin{figure}[H]
\centering
   \begin{tikzcd}
       0 \ar{r} & C_3 \ar{r}{\partial_3} & C_2 \ar{r}{\partial_2} & C_1 \ar{r} & 0  .
   \end{tikzcd}
\end{figure} \vspace{-1ex} \noindent
Let $c_1 \colon C_1 \to H_2(\partial X_1)$ be the isomorphism which sends each 1-handle generator of $C_1$ to the element of $H_2(\partial X_1)$ represented by its belt sphere, and let $b_2 \colon C_2 \to L_{\gamma}$ be the isomorphism which sends each 2-handle to the element of $L_{\gamma} \subset H_1(\Sigma)$ represented by its attaching circle. The boundary homomorphism $\partial_2$ is expressed as follows:
\[\partial_2 (h)  = \sum_j \langle \, c_1 (h_j) , \iota_{\partial X_1}b_2 (h) \, \rangle_{\partial X_1} h_j .\] 
We fix a diffeomorphism $\phi$ of $\partial X_1 \cong \sharp^k_1 S^1 \times S^2$ such that $(\partial X_1, M_{\alpha\beta}, \partial \Sigma \times [-1,1], \Sigma_{\alpha}, \Sigma_{\beta})$ correspond to $(\sharp^{k_1} S^1 \times S^2, P \times \partial^{\pm} D \, \sharp \, (\sharp^{k-l} S^1 \times S^2), \partial P \times \partial^{\pm} D, \\ P^-, P^+)$. We define $b_1 = (\phi_*)^{-1} \circ b \circ \phi_*$ by using this $\phi$ and the homomorphism $b$ in Lemma \ref{lemma4.4}. The above equality can be replaced by 
\[\partial_2 (h)  = \sum_j \langle \, b_1 c_1 (h_j) , \iota_{M_{\alpha\beta}}b_2(h) \, \rangle_{M_{\alpha\beta}} h_j. \]
Furthermore, by Lemma \ref{lemma4.2} (2), this equality can also be replaced by
\[\partial_2 (h)  =  \sum_j \langle \, \partial'_{\alpha\beta} b_1 c_1 (h_j) , b_2 (h) \, \rangle_{\Sigma} \, h_j. \]
Let $D \colon L_{\alpha}^{\partial} \cap L_{\beta}^{\partial} \to \Hom(L_{\alpha}^{\partial} \cap L_{\beta}^{\partial}, \mathbb{Z})$ be the dual map induced by the basis $\{[\partial'_{\alpha\beta} b_1 c_1 (h_j)]\}$. If we define $f_1 = D \circ \partial'_{\alpha\beta} \circ b_1 \circ c_1$ and $f_2 = b_2$, then we have
\begin{equation} \label{chain1}
\rho \circ f_2 = f_1 \circ \partial_2 . 
\end{equation}

We separate $C_3$ into $C_3^3 \oplus C_3^2$ to define $f_3$, where $C_3^i$ is the free module generated by the 3-handles derived from $X_i$. Using this decomposition, we define an isomorphism $b_3 \colon C_3 \to H_2(M_{\gamma\alpha}) \oplus H_2(M_{\beta\gamma})$ which maps each 3-handle generator of $C_3^3$ to the element of $H_2(M_{\gamma\alpha})$ represented by its attaching sphere in $M_{\gamma\alpha}$, and likewise for the generator of $C_3^2$ and $H_2(M_{\beta\gamma})$. Let $f_3 \colon C_3 \to (L_{\alpha} \cap L_{\gamma}) \oplus (L_{\beta} \cap L_{\gamma})$ be the isomorphism obtained by composing this $b_3$ with the sum of  $\partial_{\gamma\alpha} \colon H_2(M_{\gamma\alpha}) \to L_{\alpha} \cap L_{\gamma}$ and $-\partial_{\beta\gamma} \colon H_2(M_{\beta\gamma}) \to L_{\beta} \cap L_{\gamma}$ from Lemma \ref{lemma4.2} (1). We prove that 
\begin{gather} \label{chain2}
 \pi \circ f_3 = f_2 \circ \partial_3. 
\end{gather}
These equations (\ref{chain1}), (\ref{chain2}) will show that $f_*$ is a chain isomorphism between $C$ and $C^Y$, that is, the theorem holds.

It is sufficient to show (\ref{chain2}) in regard to the generators of $C_3^3$. $M_{\gamma\alpha}$ consists of 3-dimensional 2-handles whose attaching circles are $\gamma$, $\Sigma \times D^1$ and 2-handles whose attaching circles are $\alpha$. Every embedded sphere $S \subset M_{\gamma\alpha}$ is isotopic to a sphere which intersects 2-handles whose attaching circles are $\gamma$ only in disks that are parallel to the core of a 2-handle. We focus on each 2-handle $D^2 \times D^1$ to show this.

$S \cap D^2 \times D^1$ consists of some $\Sigma_{0,t_j} \, (t_j \ge 1)$ whose boundary $\partial \Sigma_{0,t_j}$ are included in $\partial D^2 \times D^1$. If there is $\Sigma_{0,t_j}$ which is not a disk and has a boundary component that is not null-homologous in $\partial D^2 \times D^1$, we move $S$ as in the following figure so that it produces a disk which is parallel to the core and $\Sigma_{0,t_j}$ whose non null-homologous boundary components decrease. 
\begin{figure}[H]
\centering
\includegraphics[scale=0.8,pagebox=cripbox]{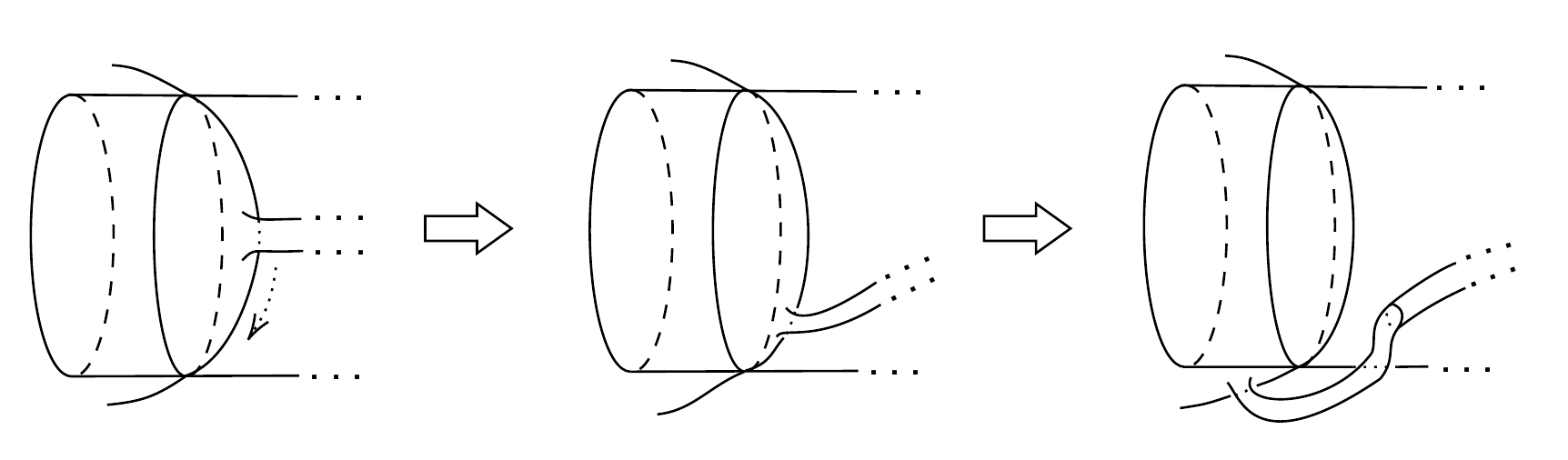}
\caption{Isotopy of $S$.}
\end{figure}
If necessary, by performing this operation several times, $S \cap D^2 \times D^1$ becomes some disks which are parallel to the core and some $\Sigma_{0,t_j}$ whose boundary components are null-homologous. Since we can regard the latter as the boundary of a tubular neighborhood of an embedded graph whose boundary components are included in $\partial D^2 \times D^1$, these can be disjoint from $\{0\} \times D^1$. Therefore, for sufficiently small $\epsilon > 0$, every $\Sigma_{0,t_j}$ is disjoint from $D_{\epsilon}^2 \times D^1$. If we replace a 2-handle $D^2 \times D^1$ with $D_{\epsilon}^2 \times D^1$, $S$ intersects the 2-handle only in disks that are parallel to the core. This shows that $S$ is isotopic to a sphere which satisfies the condition.

Performing the same operation regarding $\alpha$-side, we can assume that $S$ intersects the both sides only in the parallel disks. Furthermore, $S$ is isotopic to $S_{\gamma}$ which intersects $C_{\gamma}$ in the parallel disks. ($S$ is also isotopic to $S_{\alpha}$ in the same way.) 

Let $S$ be the attaching sphere of a 3-handle $h$. Since $f_3(h) = \partial_{\gamma\alpha} ([S])$ and $\pi$ is an inclusion, we have
\[ \pi \circ f_3(h) = [S_{\gamma} \cap \Sigma]. \]
Let $D_{ij}$ be a disk component of $S_{\gamma} \cap C_{\gamma}$ which is parallel to the core disk of 2-handle $h_j$, and then  $[S_{\gamma} \cap \Sigma] = \sum_{i,j} [\partial D_{ij}]$. From the way to define the orientation of $\partial D_{ij}$, $[\partial D_{ij}]$ is equivalent to $\langle D_{ij}, (\{0\} \times D^1)_j \rangle_{M_{\gamma\alpha}} f_2(h_j)$. Since $S_{\gamma} \cap h_j = \coprod_i D_{ij}$, we have 
\[ \sum_{i,j} [\partial D_{ij}] = \sum_j \langle S_{\gamma}, (\{0\} \times D^1)_j \rangle_{\partial Y} f_2(h_j). \]
The belt sphere of $h_j$ appears as $(\{0\} \times D^1)_j$ in $M_{\gamma\alpha}$. Therefore, the right-hand side of the formula is equal to $f_2 \circ \partial_2 (h)$.
\end{proof}

\begin{theorem}
The homology of $X$ can also be obtain from the following chain complex $C^Z$: \vspace{-1ex}
{\small \begin{figure}[H]
\begin{tikzcd}
  0 \ar{r} &[-1.5em] (L_{\alpha} \cap L_{\beta}) \oplus (L_{\beta} \cap L_{\gamma}) \oplus (L_{\gamma} \cap L_{\alpha}) \ar{r}{\zeta} &[-1.2em] L_{\alpha} \oplus L_{\beta} \oplus L_{\gamma} \ar{r}{\iota} &[-1.2em] H_1(\Sigma) \ar{r}{0} &[-1.2em] \mathbb{Z} \ar{r} &[-1.5em] 0,
\end{tikzcd}
\end{figure}} \vspace{-2ex} \noindent
where $\zeta (x, y, z) = (x-z, y-x, z-y)$ and $\iota$ is a homomorphism induced by the inclusions $\iota_{\nu}$.
\end{theorem}
\begin{proof}
Let $(C'_*, \partial'_*)$ be the chain complex of $Z$. Its nontrivial part is \vspace{-1ex}
\begin{figure}[H] 
\centering
   \begin{tikzcd}
       0 \ar{r} & C'_3 \ar{r}{\partial'_3} & C'_2 \ar{r}{\partial'_2} & C'_1 \ar{r} & 0  .
   \end{tikzcd}
\end{figure} \vspace{-1ex}
Note that $Z^1$ is diffeomorphic to $\Sigma \times D$. Applying $p = g, \, k = 2g+b-1$ to Lemma \ref{lemma4.2} and 4.3, we obtain isomorphisms $b'_1 \colon H_2(\partial Z^1) \to H_2(\Sigma \times \partial^{\pm} D, \partial (\Sigma \times \partial^{\pm} D))$ and $\partial' \colon H_2(\Sigma \times \partial^{\pm} D, \partial (\Sigma \times \partial^{\pm} D)) \to H_1(\Sigma, \partial \Sigma)$. If we define $f'_1$ and $f'_2$ in the same way as in the proof of Theorem \ref{theorem1}, the following formula also holds:
\[\iota \circ f'_2 = f'_1 \circ \partial'_2. \]

Similarly, we obtain $f'_3$ by composing $b'_3 \colon C'_3 \to H_2(M_{\alpha\beta}) \oplus H_2(M_{\beta\gamma}) \oplus H_2(M_{\gamma\alpha})$ and $\partial_{\alpha\beta} \oplus \partial_{\beta\gamma} \oplus \partial_{\gamma\alpha}$. Let $S$ be the attaching sphere of a 3-handle $h \in C^{\prime 1}_3$, and then we have
\[\zeta \circ f'_3(h) = (\partial_{\alpha\beta}([S]), -\partial_{\alpha\beta}([S]), \,\, 0 \,\, ). \]
$S$ is isotopic to $S_{\alpha}$ and $S_{\beta}$ as in the proof of Theorem \ref{theorem1}. Let $D^{\alpha}_{ij}$ and $D^{\beta}_{ij}$ be a disk component of $S_{\alpha} \cap C_{\alpha}$ and $S_{\beta} \cap -C_{\beta}$, respectively. From the way to orient $\partial_{\alpha\beta}$, the right-hand side is equal to $(\sum_{i,j} [\partial D_{ij}^{\alpha}], \sum_{i,j} [\partial D_{ij}^{\beta}], \,\, 0 \,\,)$. Even though the belt sphere of 2-handle $h_j^{\beta}$ appears as $-(\{0\} \times D^1)_j^{\beta}$ in $M_{\alpha\beta}$, the intersection number of the core disk of $h_j^{\beta}$ with it in $-C_{\beta}$ is $+1$. Therefore, we obtain 
\[ \zeta \circ f'_3 = f'_2 \circ \partial'_3. \qedhere \]
\end{proof}
Using Theorem \ref{theorem1} and \ref{theorem2}, we can describe each $H_1(X)$, $H_2(X)$ and $H_3(X)$ by $L_{\nu}$ in the same way as \cite{FKSZ1} and \cite{FM1}. We give a precise proof of the description of $H_2(X)$ by using Theorem \ref{theorem1}.
\begin{corollary} \label{corollary5.1}
 \[ H_2(X) \simeq \frac{L_{\gamma} \cap (L_{\alpha} + L_{\beta})}{(L_{\gamma} \cap L_{\alpha}) + (L_{\gamma} \cap L_{\beta})} . \]
\end{corollary}
\begin{proof}
We should calculate $\ker \rho$ and $\im \pi$. It is easy to show $\im \pi = (L_{\gamma} \cap L_{\alpha}) + (L_{\gamma} \cap L_{\beta})$ and $\ker \rho = L_{\gamma} \cap (L_{\alpha}^{\partial} \cap L_{\beta}^{\partial})^{\bot}$. Therefore, it is sufficient to show $(L_{\alpha}^{\partial} \cap L_{\beta}^{\partial})^{\bot} = L_{\alpha} + L_{\beta}$.

$(L_{\alpha}^{\partial} \cap L_{\beta}^{\partial})^{\bot} \supset L_{\alpha} + L_{\beta}$ is clear because of Lemma \ref{lemma4.1} (3). Since $(\Sigma; \alpha, \beta)$ is equivalent to $(\Sigma; \delta^{k_1}, \epsilon^{k_1})$, performing handle slides on each $\alpha$ and $\beta$ if necessary, we obtain a family of proper arcs $a$ which is disjoint from $\alpha \cup \beta$, and cuts each $\Sigma_{\alpha}$ and $\Sigma_{\beta}$ into a disk. Moreover, we can assume $[\alpha_i ] = [\beta_i] \in L_{\alpha} \cap L_{\beta} \, \, (i \le k-l)$,$\langle [ \alpha_i ] , [ \beta_j ]\rangle_{\Sigma} = \delta_{ij} \, \, (i , j \ge k-l+1 )$. For this $\alpha$, let $\{[\alpha], [\alpha^*], [a^*]\}$ be the basis as in the proof of Lemma \ref{lemma4.1} (3). We can also assume $[\alpha^*_i] = [\beta_i] \, \, (i \ge k-l+1)$. Since $\{ [\alpha_1] , \dots , [\alpha_{k-l}] , [a_1] , \dots , [a_l] \}$ becomes a basis of $L_{\alpha}^{\partial} \cap L_{\beta}^{\partial}$, $x \in (L_{\alpha}^{\partial} \cap L_{\beta}^{\partial})^{\bot}$ must be represented as a liner combination of $[\alpha_1] , \dots , [\alpha_{g-p}] , [\alpha^*_{k-l+1}] , \dots , [\alpha^*_{g-p}]$.
\end{proof}
Since the proof of the following corollary is the same as \cite{FKSZ1} and \cite{FM1}, we only introduce the forms of $H_1$, $H_3$.
\begin{corollary}
\[ H_1(X) \simeq \frac{H_1(\Sigma)}{L_{\alpha} + L_{\beta} + L_{\gamma}} , \quad  H_3(X) \simeq L_{\alpha} \cap L_{\beta} \cap L_{\gamma} . \]
\end{corollary}
\vspace{1.5ex}

To consider the intersection form of $X$, we define a bilinear form $\Phi$.
\begin{definition}
Let $\Phi \colon L_{\gamma} \cap (L_{\alpha} + L_{\beta}) \times L_{\gamma} \cap (L_{\alpha} + L_{\beta}) \to \mathbb{Z}$ be the bilinear form given as follows;
\[ \Phi (x , y) =-\langle x' , y \rangle_{\Sigma}, \]
where $x'$ is any element of $L_{\alpha}$ which satisfies $x-x' \in L_{\beta}$.
\end{definition}
There is an element $x'$ satisfying the above condition since $x \in L_{\alpha} + L_{\beta}$, but it is not unique in general.. For another $x''$, $x'-x'' \in L_{\alpha} \cap L_{\beta}$. Since $y \in (L_{\alpha}^{\partial} \cap L_{\beta}^{\partial})^{\bot}$, $\langle x'-x'' , y \rangle_{\Sigma} = 0$. Therefore, the value of $\Phi$ does not depend on a choice of $x'$. We describe the bilinear form induced by $\Phi$ on any quotient group as also $\Phi$.

The intersection form of $X$ is equivalent to $\Phi$. We can prove the following proposition by using Lemmas in Section 4 in the same way as \cite{FKSZ1}. Furthermore, we can also prove it in the same way as \cite{FM1} since this intersection is included in $\Int Z$.
\begin{proposition}
There is an isomorphism 
\[ (H_2(X) , \langle - , - \rangle_X ) \simeq \left( \frac{L_{\gamma} \cap (L_{\alpha} + L_{\beta})}{(L_{\gamma} \cap L_{\alpha}) + (L_{\gamma} \cap L_{\beta})} , \Phi \right). \]
\end{proposition}
\vspace{1ex}

\section{Representative of the Stiefel-Whitney class}
As an application of chain complexes $C^Y$ and $C^Z$, we describe representatives of the second Stiefel-Whitney class by using the matrices $_{\mu}Q_{\nu}$. Since we can convert a trisection diagram to a Kirby diagram via handlebodies, the following theorem shown by Gompf and Stipsicz \cite{GS1} plays an important role to describe the representative.
\begin{theorem*}{(Gompf-Stipsicz \cite{GS1})}
For an oriented handlebody $X$ given by a Kirby digram in dotted circle notation, $w_2(X) \in H^2(X ; \mathbb{Z}_2)$ is represented by the cocycle $c \in C^2(X ; \mathbb{Z}_2)$ whose value on each 2-handle is its framing coefficient modulo 2.
\end{theorem*}

In regard to the trisection, we have:
\begin{theorem}
Suppose that $(\alpha,\beta)$ is diffeomorphism equivalent to $(\delta^{k_1} , \epsilon^{k_1})$. Then, as a representative of the second Stiefel-Whitney class $w_2$, we can obtain $c \colon L_{\gamma} \to \mathbb{Z}_2$ defined by the following formula with respect to the $\gamma$-basis
\[ c(x) = \sum_{i=1}^{g-p} \left( _{\gamma}Q_{\beta , \partial} \, \, R^g_{p,b} \hspace{1ex} \, ^{\alpha}_{\partial}Q_{\gamma} \right)_{i \, i} \,  x_i \pmod2 . \]
\end{theorem}
\begin{proof}
Let $a$ be a family of arcs in $\Sigma$ as in the proof of Corollary \ref{corollary5.1}. From the assumption and Lemma \ref{lemma3.1}, we can draw a Kirby diagram based on the trisection diagram as follows: $\Sigma \times D^1$ is embedded in $\mathbb{R}^3 \subset S^3$, where $S^3$ is the boundary of the 0-handle. Dotted circles for 1-handles are $\partial(a_i \times D^1)$ and $\alpha_i$ which is parallel to $\beta_i$. Attaching circles of 2-handles are $\gamma$, and its framings are induced by $\Sigma$. $\partial X_1$ is a union of $C_{\alpha}$, $C_{\beta}$ and $\overline{\partial(P \times D) \setminus P \times \partial^{\pm}D}$. Since $P \times \partial^-D \, \cup \, (P \times \partial^+D \cup \overline{\partial(P \times D) \setminus P \times \partial^{\pm}D})$ is a Heegaard splitting, we can construct $\partial X_1$ from $C_{\alpha}$ by attaching a 3-dimensional 1-handlebody whose system consists of $\beta$ and $\partial(a_i \times [-1,0])$ along $\partial C_{\alpha}$. Note that $\partial C_{\alpha}$ is diffeomorphic to $\Sigma_{g+p+b-1}$, and $C_{\alpha}$ itself is a 1-handlebody whose system consists of $\alpha$ and $\partial(a_i \times [-1,0])$.

Applying the above Theorem to this Kirby diagram, we obtain the representative $c \in C^2(X;\mathbb{Z})$. When we calculate a framing (or a linking number), we can ignore dotted circles for 1-handles to regard $\partial X_1 \subset S^3$. This is equivalent to changing the system $\{\partial C_{\alpha}; \beta_1, \dots , \beta_{g-p}, \partial(a_1 \times [-1,0] ), \dots , \partial(a_l \times [-1,0])\}$ into $\{\partial C_{\alpha}; \beta_1^l, \dots , \beta_{g-p}^l, a^*_1 \times \{0\}, a^*_2 \times \{-1\}, \dots , a^*_{2p-1} \times \{0\}, a^*_{2p}\times \{-1\}, a^*_{2p+1} \times \{0\}, \dots , a^*_l \times \{0\} \}$, where $\beta^l$ is a family of curves which maps to $\epsilon^l$ by the diffeomorphism. There is no need to change $C_{\alpha}$.

If we apply Lemma \ref{lemma4.3} to this decomposition of $S^3$, there is $j_{\gamma}$ for an attaching circle $\gamma$ such that $c(\gamma) = -\langle j_{\gamma}, \gamma \rangle_{\partial C_{\alpha}}$. Since $\{[\alpha], [\beta^l], [a^*]\}$ is a basis of $H_1(\Sigma)$, $\gamma$ is expressed as $\sum_i (x_i [\alpha_i] + y_i [\beta_i^l]) + \sum_j z_j [a_j^*]$. In $\partial C_{\alpha}$, $[a^*_{2k-1}]$ corresponds to $[a^*_{2k-1} \times \{0\}]$, and $[a^*_{2k}]$ does to $[a^*_{2k} \times \{-1\}] + [\partial (a_{2k-1} \times [-1,0])]$ for $1 \le k \le p$. Since we can obtain $j_{\gamma}$ as an element which satisfies $\gamma - j_{\gamma} \in \langle [\beta^l], [a^*_{2k-1} \times \{0\}], [a^*_{2k} \times \{-1\}], a^*_i \times \{0\} \rangle$, $j_{\gamma}$ is expressed as $\sum_i x_i [\alpha_i] + \sum_k z_{2k} [\partial(a_{2k-1} \times [-1,0])]$. Regarding  $j_{\gamma}$ as an element of $L_{\alpha}^{\partial}$, we have
\[j_{\gamma} = \sum_i x_i [\alpha_i] + \sum_k z_{2k} [a_{2k-1}]. \]
For this $j_{\gamma}$, we have $c(\gamma) = -\langle j_{\gamma}, \gamma \rangle_{\Sigma}$ because of $j_{\gamma} \cap \gamma \subset \Sigma$.

Let $\alpha^{\partial}$ be a basis $\{[\alpha], [a]\}$ of  $L_{\alpha}^{\partial}$. Note that we orient them to satisfy $\langle [\beta^l_i], [\alpha_j] \rangle_{\Sigma} = \langle [a_i] , [a^*_j] \rangle_{\Sigma} = \delta_{ij}$. In regard to this basis $\alpha^{\partial}$, we obtain $j_{\gamma_i} = \sum_j ( \transpose{R^g_{p,b}} \,\, ^{\beta}_{\partial}Q_{\gamma})_{ji} \, [\alpha^{\partial}_j]$.
Since the coefficient of $[\alpha^{\partial}_j]$ is equal to $-(_{\gamma}Q_{\beta , \partial} \, R^g_{p,b})_{ij}$, this equality can be replaced by
\[j_{\gamma_i} = -\sum_j (_{\gamma}Q_{\beta , \partial} \, R^g_{p,b})_{ij} \, [\alpha^{\partial}_j]. \]
For $x = \sum_i x_i [\gamma_i] \in L_{\gamma}$, $c(x)$ is $- \sum_i \langle j_{\gamma_i} , [\gamma_i] \rangle_{\Sigma} \, x_i$. Therefore, we can calculate $c(x)$ as follows:
\begin{align}
c(x) & = \sum_i \sum_j (_{\gamma}Q_{\beta , \partial} \, R^g_{p,b})_{ij} \, \langle [\alpha^{\partial}_j] , [\gamma_i] \rangle_{\Sigma} \, x_i,  \notag \\
      & = \sum_i \left( _{\gamma}Q_{\beta , \partial} \, R^g_{p,b} \hspace{0.7ex} ^{\alpha}_{\partial}Q_{\gamma} \right)_{i  i}   x_i. \notag \qedhere
\end{align}
\end{proof}

\begin{theorem}
As a representative of the second Stiefel-Whitney class, we can obtain $c' \colon L_{\alpha} \oplus L_{\beta} \oplus L_{\gamma} \to \mathbb{Z}$ defined by the following formula with respect to the $\alpha$-, $\beta$-, $\gamma$-basis
\[c'(x) = \sum_{i=1}^{3(g-p)} \left( _{\nu}Q_a \, \,  S_{g,b} \, \hspace{0.7ex} _aQ_{\nu} \right)_{i \, i} x_i \pmod{2}. \]
\end{theorem}
\begin{proof}
Let $a$ and $a^*$ be families of curves in $\Sigma$ as in the Figure 4. We can draw a Kirby diagram as in the proof of Theorem \ref{theorem3}. The main difference is that dotted circles for 1-handles are only $\partial (a_i \times D^1)$. Therefore, when we split $\partial Z^1$ to two 1-handlebodies, their systems are the same $\{\partial(\Sigma \times D^1); \partial(a_1 \times D^1), \dots , \partial(a_{2g+b-1} \times D^1)\}$. We should change the one of the two systems into $\{\partial(\Sigma \times D^1); a^*_1 \times \{1\}, a^*_2 \times \{-1\}, \dots , a^*_{2g-1} \times \{1\}, a^*_{2g} \times \{-1\}, a^*_{2g+1} \times \{1\}, \dots , a^*_{2g+b-1} \times \{1\} \}$ to calculate a framing. 

Note that each $\{[a]\}$ and $\{[a^*]\}$ is a basis of $H_1(\Sigma, \partial \Sigma)$ and $H_1(\Sigma)$, respectively, and these satisfy $\langle [a_i], [a^*_j] \rangle_{\Sigma} = \delta_{ij}$. For $\nu_i$, we obtain $j_{\nu_i} = \sum_j (\transpose{S_{g,b}} \, _{a}Q_{\nu})_{ji} [a_j]$ regarding a basis $\{[a]\}$. Moreover, by simple calculation, this equality is equivalent to
\[ j_{\nu_i} = -\sum_j (_{\nu}Q_a \, S_{g,b})_{ij} \, [a_j]. \]
For $x = \sum_i x_i [\nu_i] \in L_{\alpha} \oplus L_{\beta} \oplus L_{\gamma}$, we can calculate $c(x)$:
\begin{align}
c(x) & = - \sum_i \langle j_{\nu_i} , [\nu_i] \rangle_{\Sigma} \, x_i,   \notag  \\
      & = \sum_i (_{\nu}Q_a \, S_{g,b} \, _{a}Q_{\nu})_{ii} \, x_i.  \notag \qedhere
\end{align}
\end{proof}

The second Stiefel-Whitney class is strongly related to the existence of spin structure of a 4-manifold $X$. Therefore, as corollaries, we have:
\begin{corollary}
Suppose that $(\alpha,\beta)$ is diffeomorphism equivalent to $(\delta^{k_1} , \epsilon^{k_1})$. Then, a trisected 4-manifold $X$ admits a spin structure if and only if there exists $d \in L^{\partial}_{\alpha} \cap L^{\partial}_{\beta}$ such that $\langle d , x \rangle_{\Sigma} = c(x) \pmod{2}$ for all $x \in L_{\gamma}$.
\end{corollary}
\begin{proof}
The first cochain group $(C^Y)^1$ is canonically isomorphic to $L_{\alpha}^{\partial} \cap L_{\beta}^{\partial}$. Let $\phi \colon L_{\alpha}^{\partial} \cap L_{\beta}^{\partial} \to (C^Y)^1$ be this isomorphism. We define another coboundary map $d'^1 \colon L_{\alpha}^{\partial} \cap L_{\beta}^{\partial} \to \Hom(L_{\gamma}, \mathbb{Z})$ by $d'^1(d) = \langle d , - \rangle_{\Sigma}$. It is clear that this is a coboundary map. Moreover, for $x \in L_{\gamma}$, we have
\[ (d^1 \phi (d))(x) = \langle d , x \rangle_{\Sigma}. \]
Therefore $((C^Y)^1, d^1)$ is isomorphic to $(L_{\alpha}^{\partial} \cap L_{\beta}^{\partial}, d'^1)$.
\end{proof}
\begin{corollary}
A trisected 4-manifold $X$ admits a spin structure if and only if there exists $d \in H_1(\Sigma, \partial \Sigma)$ such that $\langle d , x \rangle_{\Sigma} = c'(x) \pmod{2}$ for all $x \in L_{\alpha} \oplus L_{\beta} \oplus L_{\gamma}$.
\end{corollary}
\begin{proof}
$(C^Z)^1$ is isomorphic to $H_1(\Sigma, \partial \Sigma)$ by the universal coefficient theorem and the Poincar\'e duality. Let $\phi \colon (C^Z)^1 \to H_1(\Sigma, \partial \Sigma)$ be this isomorphism, and then we can show this corollary in the same way as Corollary 6.1.
\end{proof}
\vspace{1ex}

Calculating $- \langle j_{\nu_i} , \nu_j \rangle_{\Sigma}$ in the same way as the proofs of Theorem \ref{theorem3} and \ref{theorem4}, we obtain the linking matrices regarding to framed links of 2-handles.
\begin{proposition} \label{proposition6.3}
Each linking matrix of $X$ regarding to handle decompositions $Y$ and $Z$ is $_{\gamma}Q_{\beta , \partial} \, R^g_{p,b} \hspace{1ex} ^{\alpha}_{\partial}Q_{\gamma}$ and  $_{\nu}Q_a \,  S_{g,b} \, \hspace{0.3ex} _aQ_{\nu}$, respectively.
\end{proposition}
\vspace{1ex}

Finally, we compute the second Stiefel-Whitney class of the $D^2$ bundle over $S^2$ with Euler number $-1$ which admits the following diagram.
\begin{figure}[H]
\centering
\includegraphics[scale=0.9,pagebox=cripbox]{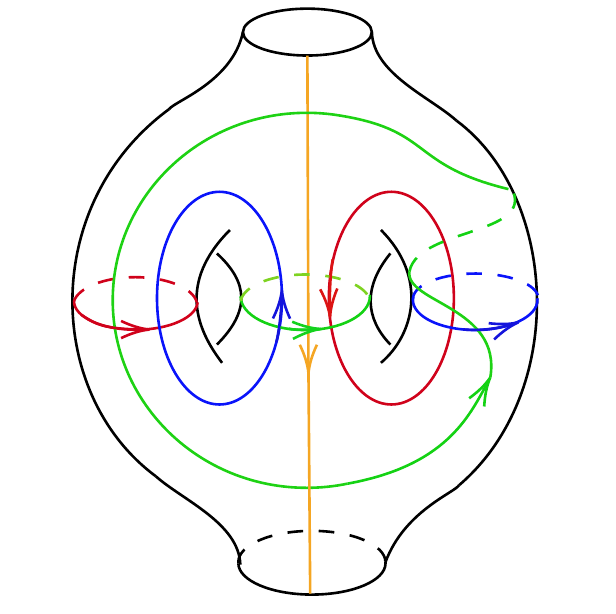}
\caption{A $(2,1;0,2)$-relative trisection diagram for the $D^2$ bundle over $S^2$ with Euler number $-1$, where the central green curve is $\gamma_1$.}
\end{figure} \noindent
We obtain the following intersection matrix;
\begin{gather}
   {}_{\gamma}Q_{\beta} = \begin{pmatrix} 1 & 0 \\ 0 & -1 \end{pmatrix} , \quad {}_{\alpha}Q_{\gamma} = \begin{pmatrix} 0 & -1 \\ 1 & 1 \end{pmatrix} ,  \notag \\[1ex]
   {}_{\beta}Q_{\alpha} = I_2 , \quad {}_aQ_{\gamma} = \begin{pmatrix} 1 & 0 \end{pmatrix}. \notag
\end{gather}
Using Proposition \ref{proposition6.3}, we can compute a linking matrix:
\[_{\gamma}Q_{\beta , \partial} \, R^g_{p,b} \hspace{1ex} ^{\alpha}_{\partial}Q_{\gamma} = \begin{pmatrix} 0 & -1 \\ -1 & -1 \end{pmatrix}. \]
Therefore, a representation matrix of $c$ regarding to the basis $\{[\gamma_1], [\gamma_2] \}$  is $(0 \,\, -1)$. Moreover, since $L_{\alpha}^{\partial} \cap L_{\beta}^{\partial}$ has a basis $\{[a_1]\}$,and $_{a}Q_{\gamma}$ is equal to $(1 \,\, 0)$, this disk bundle does not admit a spin structure.

\end{document}